\documentclass[12pt, article]{amsart}

\usepackage{color}
\usepackage{geometry}

\usepackage{graphicx, epstopdf}
\usepackage{cases}

\usepackage{enumerate}
\usepackage{bbm}
\usepackage{amssymb}
\usepackage{amscd}
\usepackage{amsfonts,latexsym,amsmath,amsxtra,mathdots,amssymb,latexsym,mathtools}
\usepackage{mathabx}
\usepackage[all,cmtip]{xy}

\usepackage{colordvi}
\usepackage{multicol}
\usepackage{hyperref,dirtytalk}
\usepackage{tikz-cd}
\usepackage{float}
\usepackage{setspace, floatflt}
\usepackage{tensor}
\usepackage[normalem]{ulem}

\allowdisplaybreaks

 \newcommand{\BC}{{\mathbb {C}}}

 \newcommand{\BN}{{\mathbb {N}}}

 \newcommand{\BZ}{{\mathbb {Z}}}

\newcommand{\GL}{{\mathrm {GL}}} 
\newcommand{\SL}{{\mathrm {SL}}} 
\newcommand{\SO}{{\mathrm{SO}}}
\renewcommand{\O}{{\mathrm{O}}}
\newcommand{\U}{{\mathrm{U}}}
 
\newcommand{\Sp}{{\mathrm{Sp}}}

\newcommand{\Gal}{\mathrm{Gal}}

\newcommand{\llc}{\mathrm{LLC}}
\newcommand{\we}{\mathrm{W}}

 \newcommand{\Ind}{\mathrm{Ind}} 
\newcommand{\Ad}{\mathrm{Ad}}

\renewcommand{\lq}{{\mathrm{LQ}}}
\newcommand{\Fr}{{\mathrm{Fr}}}
\newcommand{\WD}{{\mathrm{WD}}}

\def\-{^{-1}}

\def\1{\mathbf{1}}

\def\St{\mathrm{St}}

\def\diag{\mathrm{diag}}

\makeatletter
\g@addto@macro\normalsize{\setlength\abovedisplayskip{3pt}}
\makeatother

\makeatletter
\g@addto@macro\normalsize{\setlength\belowdisplayskip{3pt}}
\makeatother

\newcommand{\delete}[1]{}

\theoremstyle{plain}

\newtheorem{thm}{Theorem}[section] 

\newtheorem{st}[thm]{Statement}
\newtheorem*{st*}{Statement}
\newtheorem*{q*}{Question}

\newtheorem*{thm*}{Theorem}
\newtheorem*{lm*}{Lemma}
\newtheorem*{cor*}{Corollary}
\newtheorem*{rem*}{Remark}
\newtheorem{lem}[thm]{Lemma}  \newtheorem{prop}[thm]{Proposition}

\newtheorem{df}[thm]{Definition}

\newtheorem {conj}[thm]{Conjecture} 

\newtheorem {rem}[thm]{Remark}

\numberwithin{equation}{section}

\begin{document}

	\title{Local converse theorems and Langlands parameters} 
	
	\author{N. Matringe}
	\address{Nadir Matringe. Institute of Mathematical Sciences, NYU Shanghai, 3663 Zhongshan Road North Shanghai, 200062, China and
	Institut de Math\'ematiques de Jussieu-Paris Rive Gauche, Universit\'e Paris Cit\'e, 75205, Paris, France}
	\email{nrm6864@nyu.edu and matringe@img-prg.fr}

	\maketitle
	
\vspace{-0.5cm}	
	
\begin{abstract} 
Let $F$ be a non Archimedean local field, and $G$ be the $F$-points of a connected quasi-split reductive group defined over $F$. In this paper we prove a converse theorem statement for generic Langlands parameters of $G$ when $G$ is $F$-split and the Langlands dual group of $G$ is acceptable, and propose a possible extension of it without the split assumption. We also prove that the statement does not apply to $\SO_{2n}(F)$ for certain choices of $F$, as soon as $n\geq 3$. Then we consider a variant which we prove for $G=\mathrm{G}_2(F)$ and all quasi-split classical groups. When $F$ has characteristic zero and assuming the validity of the Gross-Prasad and Rallis conjecture, this latter variant translates via the generic local Langlands correspondence of Jantzen and Liu, into the usual local converse theorems for classical groups expressed in terms of Shahidi's gamma factors. 
\end{abstract}

\section{Introduction}

Local converse theorems aim to determine the isomorphism class of generic representations of local quasi-split reductive groups via their twisted gamma factors. It seems that the first instance of a local converse theorem appears in \cite[Corollary 2.19]{JL} for $\GL_2$ over a non Archimedean local field (in the same book an Archimedean converse theorem is also proved). The next converse theorem is for $\GL_3$, and it is proved in \cite{JPSScv}. The idea of the proof in \cite{JPSScv} is actually general, and with an extra argument, it was extended to $\GL_n$ by Henniart in \cite{Hcv}. Later, the so called Jacquet conjecture (one out of many), which is the sharp version of Henniart's converse theorem, was independently proved in \cite{JLiu} and \cite{Chai}. Local converse theorems for classical groups and other groups have been extensively studied since Henniart's result, and we refer to \cite{JiSo} and the more recent works \cite{Mor}, \cite{ZQ1}, \cite{ZQ2}, \cite{HzltLiu}, \cite{HKK} and \cite{Jocv}. We also mention that they are the local counterpart of global converse theorems and we refer to \cite{CPScv} for a famous instance as well as the history of such results. 

The original aim of this paper is to analyze the local statements from the perspective of Langlands parameters, hoping that it will shed light on them. It turns out that the notion of acceptability, introduced by Larsen in \cite{Lars1}, plays a key role in our result. A   complex reductive group is acceptable if two locally conjugate homomorphisms from a finite group to it are automatically globally conjugate. This notion seems to have been introduced in \cite{Lars1} with the idea of applying it to Galois groups of global fields extensions, in order to explain failure of multiplicity one phenomena in the theory of automorphic representations. 

It might be the case that the content of our paper is known to experts, but it does not seem to appear in the literature so far. In particular one of its basic idea is actually used by Gan and Savin in \cite{GS} to characterize the local Langlands correspondence for the group $\mathrm{G}_2(F)$. One can also find examples of local converse theorems for parameters in other papers of Gan and his collaborators (\cite{GS0}, \cite{GI} and \cite{GA}), but also in \cite{DHKM}, and in \cite{AT} for Archimedean parameters of $\GL_n$, using different ideas. Finally we mention that some general conjectures on local converse theorems have been made in \cite{Jgamma}, and the reader can compare them with our statements. 

We now briefly explain the main observations that are made in this paper. We denote by $G$ be the $F$-points of a connected quasi-split reductive group defined over $F$, Weil-Deligne group $\WD_F=\we_F\times \SL_2(\BC)$ and L-group ${}^LG=G^\vee\ltimes \we_F$. First, thanks to the Local Langlands correspondence for $\GL_n(F)$ (\cite{HTllc} and \cite{Hllc}), the local converse theorem of Henniart or its sharp refinement implies that if $\phi:\WD_F\to {}^L G$ is any admissible homomorphism with associated Weil-Deligne parameter $(\varphi,N)$ (see Section \ref{sec gen}), and $r:{}^L G\to \GL(V)$ is an irreducible representation of ${}^L G$, then the twisted gamma factors of $r\circ \phi$ determine the $\GL(V)$-conjugacy class of 
$r\circ \varphi$ (see Section \ref{sec cons}). Suppose that this holds for any irreducible $r$, then, when $G$ is $F$-split, a classical result of Steinberg implies that the conjugacy class of $\varphi(\gamma)$ by the Langlands dual group $\widehat{G}$ is uniquely determined for any $\gamma\in \we_F$. Under the assumption that $\widehat{G}$ is acceptable, using in particular results of Yu on acceptability (Section \ref{sec acc}) and of Silberger and Zink on the Langlands classification of parameters (Section \ref{sec SZacc}), we prove in Theorem \ref{thm acc implies wacc} that this property in turn implies that the equivalence class of $\varphi$ is uniquely determined. If $\phi$ is moreover assumed to be generic this determines the class of $\phi$ uniquely, this folklore fact having recently been proved rigorously in \cite{DHKM} as we recall in Section \ref{sec gen}. As a summary of this discussion, we state the key result of this paper. Say that $G$ is $\we_F$-acceptable if whenever we have parameters $\varphi,\varphi': \we_F \to \widehat{G}$ such that $r \circ \varphi$ is
isomorphic to $r \circ \varphi'$ for all algebraic representations $r$ of $\widehat{G}$, then $\varphi$ is isomorphic to $\varphi'$. Say that $G$ is generically $\WD_F$-acceptable when the above property holds whenever the parameters $\varphi,\varphi': \WD_F \to \widehat{G}$ are generic.  

\begin{thm*}
Suppose that $\widehat{G}$ is acceptable, then $G$ is $\we_F$-acceptable, hence it is generically $\WD_F$-acceptable. 
\end{thm*}

This leads to our main result Theorem \ref{thm main}, to which we refer for more complete statements. We also prove that the statement of Theorem \ref{thm main} does not apply to the group $\SO_{2n}(F)$ for some specific fields $F$ as soon as $n\geq 3$, by studying unacceptable parameters in Section \ref{sec unacc}. An interesting question, which woud require a case by case analysis of unacceptable parameters but which can be addressed thanks to the methods of this paper, is whether local converse theorems for split groups only hold when the Langlands dual group is acceptable. 
Finally using ideas similar to those of our main theorem, we prove the variant which coincides with the known local converse theorems via the generic local Langlands correspondence in Theorem \ref{thm variant}.

\subsection*{Acknowledgement.} We thank the referee for many useful comments. We thank Moshe Adrian, Guy Henniart, Wee-Teck Gan, Rob Kurinczuk and Shaun Stevens for useful correspondence. We especially thank Ahmed Moussaoui for answering many questions about this work, and Eyal Kaplan for very useful correspondence and detailed comments on an earlier draft, allowing significant improvement of its exposition. Finally and most importantly, we thank Dipendra Prasad for many useful answers and bringing the notion of acceptability to our attention, as well as for spotting a mistake in a previous version, the correction of which allowed to strengthen our results.

\section{Preliminaries} 

All representations that we consider have complex coefficients. The main reference for the material in this section are \cite{Tcor} and \cite{Bcor}, and \cite[Section 6]{DHKM}. Let $F$ be a non Archimedean local field with residual cardinality $q$, Weil group $\we_F$ and Weil-Deligne group $\WD_F=\we_F\times \SL_2(\BC)$. We do not impose any restriction on the characteristic of $F$. We denote by $I_F$ the inertia subgroup of $\we_F$, and fix $\Fr_F$ a geometric Frobenius element of $\we_F$. We set $\nu_F:\we_F\to \BC^\times$ to be the unique character trivial on $I_F$ such that $\nu_F(\Fr_F)=q^{-1}$. Via local class field theory, the map $\nu$ corresponds to the normalized absolute value $|\ |$ on $F^\times$. For $\gamma\in \we_F$ we denote by $d(\gamma)$ the integer such that $\nu(\gamma)=q^{-d(\gamma)}$. Let $G$ be the group of $F$-points of a quasi-split connected reductive group defined over a local field $F$. We denote by $\widehat{G}$ its Langlands dual group and its $L$-group by ${}^LG:=\widehat{G} \rtimes \we_F$, where $\we_F$ acts trivially on $\widehat{G} $ if $G$ is split over $F$. For example ${}^L\GL_n(F)= \GL_n(\BC)\times \we_F$. By definition, what we call a representation of ${}^LG$ is a finite dimensional semi-simple and continuous representation.

\subsection{Langlands parameters and their local constants}
 
We say that a continuous homomorphism $\varphi:\we_F\to {}^LG$ is \textit{semi-simple} if the projection of its image on $\widehat{G}$ consists of semi-simple elements, and the second coordinate of $\phi$ is the identity of $\we_F$. We say that a homomorphism $\phi:\WD_F\to {}^LG$ is \textit{admissible} if $\phi_{|\we_F}$ is semi-simple, and $\phi_{|\SL_2(\BC)}$ is algebraic. Let $X$ be a group and $\tau_1,\ \tau_2:X\to {}^LG$ two homomorphisms, we write 
$\tau_1\underset{L}{\sim} \tau_2$ if they are conjugate under ${}^LG$. We say that $\tau_1$ and $\tau_2$ are \textit{equivalent} and write $\tau_1 \sim \tau_2$ if they are conjugate under $\widehat{G}$. An admissible homomorphism 
$\phi:\WD_F\to {}^L\GL_n(F)$ identifies with a (smooth) semi-simple representation of $\WD_F$ of dimension $n$, because composition with the projection onto 
$\GL_n(\BC)$ is injective, and we will do this identification freely. 

We will also need another version of admissible homomorphisms, which is that of \textit{Weil-Deligne parameters} as in \cite[Definition 6.1]{DHKM}. We denote by $\widehat{\mathfrak{g}}$ the Lie algebra of $\widehat{G}$, and by $\Ad :{}^LG\to \GL(\widehat{\mathfrak{g}})$ the adjoint representation of ${}^LG$. A Weil-Deligne parameter consists of a pair $(\tau, N)$, where $\tau:\we_F\to {}^L G$ is \textit{semi-simple homomorphism}, and $N$ is a nilpotent element in $\widehat{\mathfrak{g}}$ satisfying \begin{equation} \label{eq mono} \Ad(\tau(\gamma))(N)=\nu(\gamma)N\end{equation} for all $\gamma \in \we_F$. We denote by $\mathcal{N}_\tau$ the space of (necessary nilpotent) elements in $\widehat{\mathfrak{g}}$ satisfying Equation \eqref{eq mono}, and by $\widehat{G}_\tau$ the centralizer of the image of $\tau$ in $\widehat{G}$. There is a natural action of $\widehat{G}$ on the set of Weil-Deligne parameters given by the formula 
\[x\cdot (\phi,N)=(x\cdot \phi, \Ad(x)(N)),\] and we say that two Weil-Deligne parameters are equivalent if they are in the same $\widehat{G}$-orbit. Let $\phi:\WD_F\to {}^L G$ be an admissible homomorphism, and denote by $d\phi$ the differential of $\phi_{|\SL_2(\BC)}$ at $I_2$. As explained in \cite[Remark 6.4, (2)]{DHKM}, one asociates to $\phi$ the Weil-Deligne parameter $(\varphi,N)$ where 
\begin{equation}\label{eq 0}\varphi(\gamma):=\phi\left(\gamma,\begin{pmatrix} q & 0 \\ 0 &  q^{-1}\end{pmatrix}^{d(\gamma)/2}\right), \ N:=d\phi\begin{pmatrix} 0 & 1 \\ 0 & 0 \end{pmatrix}.\end{equation} This association induces a bijection between equivalence classes. 

Now we discuss local constants following \cite[Section 2.2]{GR}. To a semi-simple representation $\phi$ of $\WD_F$ or more precisely to its Weil-Deligne parameter $(\varphi,N)$, one can associate thanks to the work of Artin an $L$-factor $L(s,\phi)$ which is meromorphic in the complex variable $s$. We fix $\psi:F\to \BC^\times$ a non trivial additive character. Then one can associate to $\phi$, thanks to the works of Deligne and Langlands, a gamma factor $\gamma(s,\phi,\psi)$ which is again a meromorphic function in the complex variable $s$. These local factors only depend on the isomorphism class of $\phi$, and for gamma factors we have the stronger equality
\begin{equation}\label{eq gam1} \gamma(s,\phi,\psi)=\gamma(s,\varphi,\psi).\end{equation} 

\subsection{Generic parameters}\label{sec gen}

 If $\phi:\WD_F \to {}^LG$ is an admissible homomorphism, the \textit{adjoint $L$-factor} of $\phi$ is by definition $L(s,\Ad \circ \phi)$. Following \cite[Conjecture 2.6]{GrP}, we call an admissible homomorphism $\phi:\WD_F\to {}^LG$ a \textit{generic} homomorphism if its adjoint $L$-factor is holomorphic at $s=1$. This notion only depends on the equivalence class of $\phi$. For brievety, if $\phi:\WD_F\to {}^L G$ is an admissible homomorphism, we will always denote by $\varphi$ the semi-simple part of its Weil-Deligne parameter. We have the following result.

\begin{lem}\label{lm dhkm}
Let $\phi_1,\phi_2:\WD_F\to {}^LG$ be two generic homomorphisms, then $\phi_1\sim \phi_2 \iff \varphi_1 \sim \varphi_2$. Similarly $\phi_1\underset{L}{\sim} \phi_2 \iff \varphi_1 \underset{L}{\sim} \varphi_2$.
\end{lem}
\begin{proof}
This follows from the equivalence between (1) and (2) in \cite[Proposition 6.10]{DHKM}, which asserts that an admissible homomorphism $\phi:\WD_F\to {}^LG$ is generic if and only if its corresponding Weil-Deligne parameter $(\varphi,N)$ is such that $N$ lies in the unique open $\widehat{G}_\varphi$-orbit in $\mathcal{N}_\varphi$.  
\end{proof}

We also recall that an admissible homomorphism $\phi:\WD_F\to {}^LG$ is called \textit{tempered} (or bounded) if the projection of $\phi(\we_F)$ on $\widehat{G}$ is bounded. If $\phi$ is tempered, its adjoint $L$-factor is holomorphic at any complex number with positive real part, and in particular $\phi$ is generic.

\section{A consequence of the sharp local converse theorem for $\GL_n$}\label{sec cons}

We denote by $\mathrm{sp}(n)$ the $n$-dimensional irreducible representation of $\SL_2(\BC)$. The local Langlands correspondence $\llc$ for $\GL_n(F)$ (\cite{HTllc} and \cite{Hllc}), associates to the equivalence class of the semi-simple representation $\phi:\WD_F\to \GL_n(\BC)$ an isomorphism class of smooth irreducible representations $\llc(\phi)$ of $\GL_n(F)$. Moreover it sends irreducible representations of $\WD_F$ trivial on $\we_F$ to cuspidal representations. Now we recall that if $\rho$ is a cuspidal representation of $\GL_d(F)$, one can attach to it and an integer $r\geq 1$ the generalized Steinberg representation $\St_{r}(\rho)$ of $\GL_{rd}(F)$, which is the irreducible quotient of the Bernstein-Zelevinsky product (see \cite{Zel}) \[|\det |_F^{(1-r)/2}\rho\times \dots \times |\det |_F^{(r-1)/2}\rho.\] We also denote by $e(\St_r(\rho))$ the real number such that $|\det|^{-e(\St_{r}(\rho))}\St_{r}(\rho)$ is unitary. Let $t$ be a positive integer. For each $i\in \{1,\dots,t\}$, let $r_i, \ d_i$ be  positive integers, and $\rho_i$ be a cuspidal representation of $\GL_{d_i}(F)$. It follows from \cite{Zel} that if the real numbers $e(\St_{r_i}(\rho_i))$ form a non increasing sequence, then the Bernstein-Zelevinsky product \[\St_{r_1}(\rho_1)\times \dots \times \St_{r_t}(\rho_t)\] has a unique irreducible quotient \[\lq(\St_{r_1}(\rho_1)\times \dots \times \St_{r_t}(\rho_t)).\] The following result is proved in \cite{Henniart-Characterization+LLC}.

\begin{lem}\label{lm a}
Let $\sum_{i=1}^t \varphi_i \otimes \mathrm{sp}(r_i)$ be a semi-simple representation of $\WD_F$ with $\varphi_i$ trivial on $\SL_2(\BC)$. Set $\rho_i:=\llc(\varphi_i)$ and assume that the real numbers $e(\St_{r_i}(\rho_i))$ form a non increasing sequence. Then: 
\[\llc(\sum_{i=1}^t \varphi_i \otimes \mathrm{sp}(r_i))=\lq(\St_{r_1}(\rho_1)\times \dots \times \St_{r_t}(\rho_t)).\]
\end{lem}

This has the following consequence, which is well-known, but for which we did not find a reference. 

\begin{lem}\label{lm b}
If $\pi$ is an irreducible representation of $\GL_n(F)$, then the semi-simple part $\varphi_{\pi}$ of the Weil-Deligne parameter of $\pi$ only depends on the cuspidal support of $\pi$, and is actually the sum of the parameters of the representations in its cuspidal support.
\end{lem}
\begin{proof}
Write $\pi$ as $\lq(\St_{r_1}(\rho_1)\times \dots \times \St_{r_t}(\rho_t))$ as in Lemma \ref{lm a}. Then its Langlands parameter is $\phi:=\sum_{i=1}^t \varphi_i \otimes \mathrm{sp}(r_i)$ where $\llc(\varphi_i)=\rho_i$. First, this implies that the restriction to $\we_F$ of $\phi$ is 
$\oplus_{i=1}^t (\varphi_i \oplus \dots  \oplus \varphi_i)$. Now set 
\[\varphi=\oplus_{i=1}^t (\nu^{(1-r_i)/2}\varphi_i \oplus \dots  \oplus \nu^{(r_i-1)/2}\varphi_i).\] Moreover, for each $i$, denote by $N_i$ the nilpotent endomorphism \[N_i:(v_1,\dots,v_{r_i})\to (v_2,\dots,v_{r_i},0)\] of the space $V_{\varphi_i} \oplus \dots  \oplus V_{\varphi_i}$. If one sets \[N:=\oplus_{i=1}^t N_i,\] then one checks that $(\varphi,N)$ is in the class of Weil-Deligne parameters attached to $\phi$ in Equation \eqref{eq 0}, and the result thus follows. 
\end{proof} 

The following statement is a translation via the local Langlands correspondence of the main result of \cite{Chai} and \cite{JLiu}, which is a sharp version of Henniart's converse theorem in \cite{Hcv}. 

\begin{thm}\label{thm loc cv}
Let $\phi_1$ and $\phi_2$ be two semi-simple representations of $\WD_F$ of dimension $n$. Then $\varphi_1\sim \varphi_2$ if and only if 
\[\gamma(s, \phi_1 \otimes \phi,\psi)=\gamma(s, \phi_2 \otimes \phi,\psi)\] for all irreducible representations $\phi$ of $\we_F$ of dimension at most $\lfloor n/2 \rfloor$. 
\end{thm}
\begin{proof}
The direct implication follows from Equation \eqref{eq gam1}. Let us explain the converse direction. Let $\pi_1$ be the irreducible representation of $\GL_n(F)$ parametrized by $\phi_1$, and $\pi_2$ that parametrized by $\phi_2$. We denote by $\tau_1$ and $\tau_2$ respectively, the unique generic representation of $\GL_n(F)$ with same cuspidal support as $\pi_1$ and $\pi_2$. Thanks to the preservation of local constant properties of the local Langlands correspondence, and by multiplicativity of Rankin-Selberg gamma factors (\cite{JPSS}), we have 
\[\gamma(s, \tau_1 \otimes \rho,\psi)=\gamma(s, \tau_1 \otimes \rho,\psi)\] for all cuspidal representations $\rho$ of $\GL_k(F)$ for all $1\leq k \leq \lfloor n/2 \rfloor$. Then according to \cite{Chai} and \cite{JLiu}, and by multiplicativity of gamma factors again, we deduce that 
$\tau_1\simeq \tau_2$, which is the same as saying that $\pi_1$ and $\pi_2$ have the same cuspidal support. Going back to Weil-Deligne parameters, this exactly says that $\varphi_1 \sim \varphi_2$ according to Lemma \ref{lm b}. 
\end{proof}

\section{Acceptable reductive algebraic groups over $\BC$}\label{sec acc}

First for a general group $H$ and $X$ a subset of $H$, we denote by $C_{H}(X)$ the centralizerof $X$ in $H$. Let $\widehat{G}$ be a connected reductive complex Lie group. Following \cite{Lars1} and 
\cite{Yu} we define the notion of (strong) acceptability. 

\begin{df}
The group $\widehat{G}$ is called 
\begin{enumerate}
\item acceptable, if for any finite group $\Gamma$, any two locally conjugate group homomorphisms $\varphi,\varphi':\Gamma \to \widehat{G}$ (i.e. $\varphi(\gamma)$ and $\varphi'(\gamma)$ are conjugate in $\widehat{G}$ for each $\gamma\in \Gamma$) are automatically globally conjugate (i.e. there is $g\in \widehat{G}$ such that $\phi'=g\phi g^{-1}$), 
\item strongly acceptable, if for any group $\Gamma$, any two locally conjugate group homomorphisms $\varphi,\varphi':\Gamma \to \widehat{G}$ are automatically globally conjugate.
\end{enumerate}
\end{df}

The notion of acceptability is now well understood, and we refer for example to \cite{Lars1}, \cite{Lars2}, \cite{GaCh} and \cite{Yu}. In particular, according to \cite[Propositions 1.6 and 1.7]{Lars1} and \cite[Theorem 1.1]{Yu}, we have the following result. 

\begin{thm}\label{thm str vs weak}
\begin{enumerate}
\item \label{sisi} Let $\widehat{K}$ be a maximal compact subgroup of $\widehat{G}$, then two elements of $\widehat{K}$ are $\widehat{K}$-conjugate if and only if they are $\widehat{G}$-conjugate. 
\item \label{chacal} A connected compact Lie group is acceptable if and only if it is strongly acceptable. 
\item The group $\widehat{G}$ is acceptable if and only if its unique up to conjugacy maximal compact subgroup is strongly acceptable. 
\end{enumerate}
\end{thm}

We will need the following consequence of \cite[Lemma 2.3]{Yu}. We recall that $\widehat{G}$ has up to conjugacy a unique Cartan involution, the fixed points of which form a maximal compact subgroup of $\widehat{G}$. 

\begin{prop}\label{prop crux}
Let $\widehat{G}$ be acceptable. 
\begin{enumerate}
\item If $\widehat{L}$ is a Levi subgroup of $\widehat{G}$, then $\widehat{L}$ is acceptable. 
\item If $\Gamma$ is a group, and $\varphi, \varphi':\Gamma\to \widehat{G}$ are locally conjugate homomorphisms with bounded image, then they are globally conjugate.
\end{enumerate}
\end{prop}
\begin{proof}
For the first part, we let $\theta$ be a Cartan involution of $\widehat{G}$. Up to conjugating $\widehat{L}$, we may assume that it is stabilized by $\theta$. Denote by $U$ the connected center of $\widehat{L}^{\theta}$, in this situation we have $\widehat{L}^{\theta}=C_{\widehat{G}^\theta}(U)$ and $\widehat{L}^{\theta}$ is known to be connected (see the references in \cite[Section 2]{Yu}). By Theorem \ref{thm str vs weak} the group $\widehat{G}^{\theta}$ is strongly acceptable, hence so is $\widehat{L}^{\theta}$ thanks to \cite[Lemma 2.3]{Yu}. Finally we deduce that $\widehat{L}$ is acceptable thanks to Theorem \ref{thm str vs weak} again. For the second part, both $\varphi$ and $\varphi'$ have their respective image in some possibly different maximal compact subgroup of $\widehat{G}$. Because such maximal compact subgroups are all conjugate to one another, we may assume that $\varphi$ and $\varphi'$ actually have their image in $\widehat{G}^\theta$. The result now follows from \eqref{sisi} and \eqref{chacal} of Theorem \ref{thm str vs weak}.
\end{proof}

Moreover we mention that the references \cite{Lars1}, \cite{Griess}, \cite{Lars2}, \cite{GaCh} and \cite{Yu} provide a quite complete classification of acceptable groups. 

\begin{prop}\label{prop acc gr}
A complex reductive group is admissible if and only if its derived subgroup is a product of groups in the following list: $\SL_n(\BC)$, $\Sp_{2n}(\BC)$, $\SO_{2n+1}(\BC)$, $\mathrm{G}_2(\BC)$ and $\SO_4(\BC)$.
\end{prop}

On the other hand even orthogonal $\SO_{2n}(\BC)$ for $n$ larger than $3$, are for example unacceptable.
 
\section{A consequence of results of Silberger and Zink}\label{sec SZacc}

In this section we consider $\widehat{G}$ as in Section \ref{sec acc}. Our goal is to prove Theorem \ref{thm acc implies wacc}, which states that acceptability of $\widehat{G}$ implies its "$\we_F$-acceptability". We refer to \cite[Section 5.1]{SZ} for the notions of hyperbolic and elliptic elements in a complex reductive group (\cite{SZ} more generally considers L-groups of quasi-split groups which might be disconnected, but we will not need this degree of generality here). If $\varphi:\we_F\to \widehat{G}$ is a semi-simple homomorphism (identified with an admissible homomorphism of $\WD_F$ trivial on $\SL_2(\BC)$) and $\gamma\in \we_F$, we also refer to \cite[Section 5.1]{SZ} for the definition of the hyperbolic and elliptic parts of $\varphi(\gamma)$. Now we state a consequence of \cite[Proposition 5.3]{SZ}. 

\begin{prop}\label{prop SZ1}
Let $\varphi:\we_F\to \widehat{G}$ be a semi-simple homomorphism. Then there exists a unique hyperbolic element $z(\varphi)$ in the connected centralizer of $\varphi(\we_F)$ in $\widehat{G}$, such that the semi-simple homomorphism $\varphi_t:\we_F\to  {}^LG$ defined by 
\[\varphi_t(\gamma):=\varphi(\gamma)z(\varphi)^{-d(\gamma)}\] is tempered. Actually for any $\gamma\in \we_F$, the element $\varphi_t(\gamma)$ is the elliptic part of $\varphi(\gamma)$ and $z(\varphi)^{d(\gamma)}$ its hyperbolic part. 
\end{prop}

Then by \cite[Proposition 5.5]{SZ} we have the following result.

\begin{prop}\label{prop SZ2}
The centralizer of $z(\varphi)$ in $\widehat{G}$ is a Levi subgroup of $\widehat{G}$. 
\end{prop}

We denote by $\widehat{L}_{\varphi}$ this Levi subgroup. 
 
\begin{thm}\label{thm acc implies wacc}
Suppose that $\widehat{G}$ is acceptable and let $\varphi,\varphi':\we_F\to \widehat{G}$ be locally conjugate semi-simple homomorphisms. Then they are globally conjugate.
\end{thm}
\begin{proof}
By local conjugacy, we can suppose that $\varphi(\Fr_F)=\varphi'(\Fr_F)=:f$. This means that $\varphi_t(\Fr_F)=\varphi_t'(\Fr_F)$ and $z(\varphi)=z(\varphi'):=z$. Let $\widehat{L}:=\widehat{L}_z$ be the centralizer of $z$ in $\widehat{G}$. Hence the semi-simple homomorphisms 
$\varphi_t$ and $\varphi_t'$ are $\widehat{L}$ valued, and we want to prove that they are locally $\widehat{L}$-conjugate. If 
$\gamma\in \we_F-I_F$, then the $\widehat{G}$-conjugacy of $\varphi(\gamma)=\varphi'(\gamma)$ and the uniqueness of the polar decomposition imply that they are $C_{\widehat{G}}(z^{d(\gamma)})$-conjugate, i.e. that they are $\widehat{L}$-conjugate because $C_{\widehat{G}}(z^d)=C_{\widehat{G}}(z)$ for any $d\neq 0$ by hyperbolicity of $z$. If $\gamma=\iota\in I_F$, the argument is a refinement of the previous one. First one checks that there exists $d\in \BN^*$ such that:
\begin{enumerate}
\item $f^d$ commutes with the finite groups $\varphi(I_F)$ and $\varphi'(I_F)$, 
\item for any $k\geq 1$, $C_{\widehat{G}}(f^{dk})=C_{\widehat{G}}(f^{d})$.
\end{enumerate}
Now there exists $g\in \widehat{G}$ such that $g\varphi(\Fr_F^d\iota )g^{-1}=\varphi'(\Fr_F^d\iota )$ by assumption. By considering hyperbolic parts, we deduce that $g=l\in C_{\widehat{G}}(z^d)=\widehat{L}$. Moreover exponentiating te above equality to some power $k\geq 1$ such that 
$\varphi( \iota )^k=\varphi'( \iota )^k=1$, we deduce that $g\in C_{\widehat{G}}(f^{dk})=C_{\widehat{G}}(f^{d})$. From this we deduce that 
$l\varphi(\iota )l^{-1}=\varphi'(\iota )$. Hence $\varphi_t$ and $\varphi_t'$ are locally $\widehat{L}$-conjugate and bounded, so they are globally $\widehat{L}$-conjugate thanks to Proposition \ref{prop crux}. This implies that $\varphi$ and $\varphi'$ are globally conjugate.
\end{proof}

\section{A local converse theorem for Langlands parameters}\label{sec main}

We conjecture that the following result holds.

\begin{conj}\label{conj main}
Let $G$ be a quasi-split group with acceptable Langlands dual group. For $i=1,2$, let $\phi_i:\we_F\to {}^LG$ be an admissible homomorphism with associated Weil-Deligne parameter $(\varphi_i,N_i)$. 
\begin{enumerate}
\item \label{u} Suppose that for any irreducible representation $r$ of ${}^LG$, and any irreducible representation $\phi$ of $\WD_F$, one has \[\gamma(s,r\circ \phi_1  \otimes \phi,\psi)=\gamma(s,r\circ \phi_2 \otimes \phi,\psi).\] Then $\varphi_1 \underset{L}{\sim} \varphi_2$. In particular $\phi_1\underset{L}{\sim} \phi_2$ if both $\phi_1$ and $\phi_2$ are generic (see Lemma \ref{lm dhkm}). If $G$ is $F$-split one can replace $\underset{L}{\sim}$ by $\sim$ in the two previous sentences. 
\item \label{v} If $G$ is $F$-split, semi-simple and simply connected, it is sufficient to take $r$ above amongst the finite number of fundamental representations of ${}^LG$, and to twist by irreducible representations $\phi$ of $\WD_F$ of dimension at most $d$, where $d:=\lfloor N/2 \rfloor$ for $N$ the dimension of any fundamental representation of highest dimension. 
\end{enumerate}
\end{conj} 

Before proving it in a special case, we recall the following result of result of Steinberg (\cite[Corollary 6.6]{Stb}):

\begin{thm}\label{thm st}
Let $\widehat{G}$ be the group of complex points of a connected reductive algebraic group and let $x$ and $y$ be two semi-simple elements in $\widehat{G}$. Suppose that for any irreducible representation $r$ of $\widehat{G}$, we have $\chi_r(x)=\chi_r(y)$ where $\chi_r$ is the trace character of $r$, then $x$ and $y$ are $\widehat{G}$-conjugate. Moreover if $\widehat{G}$ is semi-simple and simply connected, only the fundamental representations of $\widehat{G}$ are required in order to obtain this conclusion.  
\end{thm}
\begin{proof}
When $G$ is semi-simple this is \cite[Corollary 6.6]{Stb}. If $\widehat{G}$ is reductive we can write $x=cs$ and $y=c's'$ with $c,c'$ in the connected center $\widehat{Z}^0$ of $\widehat{G}$ and $s,s'$ in its derived subgroup $\widehat{G}^{\mathrm{der}}$. First by considering all possible one dimensional representations $r$, we conclude that $c$ and $c'$ are equal up to an element of the finite group 
$\widehat{Z}^0 \cap \widehat{G}^{\mathrm{der}}$. In other words we may assume that $c=c'$. 
Then we observe that any irreducible representation of $\widehat{G}^{\mathrm{der}}$ can be extended to an irreducible representation of $\widehat{G}$ by extending the restriction of its central character to $\widehat{Z}^0 \cap \widehat{G}^{\mathrm{der}}$, to an alegbraic character of $\widehat{Z}^0$. In particular $s$ and $s'$ are $\widehat{G}^{\mathrm{der}}$-conjugate, so $x$ and $y$ as well. 
\end{proof}

Before stating the main result of this paper, we observe that if $r:{}^LG\to \GL(V)$ is a representation, and $\phi:\WD_F\to {}^LG$ is an admissible homomorphism with associated Weil-Deligne parameter $(\varphi,N)$, then the Weil-Deligne representation associated to 
$r\circ \phi$ is $(r\circ \varphi, dr(N))$ where $dr$ is the differential of $r_{|\widehat{G}}$ at the identity. 

\begin{thm}\label{thm main}
Conjecture \ref{conj main} is true if $G$ is $F$-split. 
\end{thm} 
\begin{proof}
Because $G$ is $F$-split, we may and do replace ${}^L G$ by $\widehat{G}$ in this proof. According to Lemma \ref{lm dhkm} just need to prove that \[\varphi_1\sim \varphi_2.\] Now let $(r,V)$ be an irreducible representation of $\widehat{G}$. By Theorem \ref{thm loc cv}, under the assumption that the appropriate twisted gamma factors of $r\circ \phi_1$ and $r\circ \phi_2$ are equal, the homomorphisms $r\circ \varphi_1$ and $r\circ \varphi_2$ are $\GL(V)$-conjugate. In particular, under the hypothesis of \eqref{u} or \eqref{v} of Conjecture \ref{conj main}, Theorem \ref{thm st} then implies that $\varphi_1$ and $\varphi_2$ are locally conjugate. The result follows from Theorem \ref{thm acc implies wacc}. 
\end{proof}

\begin{rem}
The above proof could probably be extended to quasi-split $G$ if it was known that the result of Steinberg (Theorem \ref{thm st}) holds for disconnected reductive groups, or at least L-groups of connected reductive groups. There has been recent literature on representation theory of disconnected reductive groups (\cite{AHR} and \cite{JDR}), but we could not extract the extension of Steinberg's result from it. 
\end{rem}

\begin{rem}
For $G=\SO_4(F)$, Theorem \ref{thm main} is in agreement with the main result of \cite{YZ} via the generic Langlands correspondence. 
\end{rem}

We now check that Conjecture \ref{conj main} holds for stable discrete parameters of unitary groups. We will observe after the easy proof of this fact, that one can actually deduce from Theorem \ref{thm main} (which is for split groups) a stronger result. Hence Conjecture \ref{conj main} for non-split quasi-split groups, if correct, is in general not optimal, and maybe not very interesting. If $E/F$ is a separable quadratic extension, and $G:=\U_n(E/F)$ is the corresponding quasi-split group, we call an admissible homomorphism from $\WD_F$ to ${}^LG$ \textit{stable discrete} if its restriction to $\WD_E$ (interpreted as semi-simple $n$-dimensional representation) is irreducible. Such a parameter is automatically generic. 

\begin{prop}
Let $n$ be a positive integer, $E/F$ be a separable quadratic extension, and $G:=\U_n(E/F)$ the corresponding quasi-split group. Then Conjecture \ref{conj main} holds for stable discrete homomorphisms from $\WD_F$ to ${}^LG$. 
\end{prop}
\begin{proof}
We set $J=\begin{pmatrix} & & & 1 \\ & & -1 &  \\  & \iddots & &  \\ (-1)^{n-1} \end{pmatrix}$. We have $\widehat{G}=\GL_n(\BC)$ and for $g\in\GL_n(\BC)$ we write $g^{-T}$ for the inverse of the transpose of $g$. The L-group ${}^LG$ admits $\GL_n(\BC)\rtimes \BZ/2\BZ$ as a quotient, where the non trivial element $\alpha$ of $\BZ/2\BZ$ by the automorphism 
$\alpha(g)=J g^{-T} J^{-1}$ of $\GL_n(\BC)$. Let $\phi_1,\phi_2:\WD_F\to {}^LG$ be stable and discrete homomorphisms. We will prove that if $r\circ \phi_1$ and $r\circ \phi_2$ are equivalent representations of $\we_F$ for all representations $r$ of ${}^LG$, then $\phi_1\underset{L}{\sim}\phi_2$. We take $\St$ the identity representation of $\GL_n(\BC)$, and set 
\[r_0:=\Ind_{\GL_n(\BC)}^{\GL_n(\BC)\rtimes \BZ/2\BZ}(\St).\] This irreducible representation of ${}^LG$ (obtained by inflating that of its quotient $\GL_n(\BC)\rtimes \BZ/2\BZ$) is enough for our modest purpose. Under our assumption, the restriction $(\phi_i)_E$ of $\phi_i$ to $\WD_E$ is irreducible. From this it follows that $(r_0\circ \phi_i)_E=(\phi_i)_E \oplus \alpha\circ (\phi_i)_E$, hence either $(\phi_2)_E=(\phi_1)_E$, or $(\phi_2)_E=\alpha\circ (\phi_1)_E$. We deduce that $\phi_2=\phi_1$, or $\phi_2=\alpha\circ \phi_1$ thanks to for example \cite[Proposition 7.10]{Prel}. 
\end{proof}

\begin{rem}
As previously mentioned, one can do better using Theorem \ref{thm main}. See Theorem \ref{thm variant} below, Case \eqref{gr 0}. 
\end{rem}

\section{Unacceptable parameters}\label{sec unacc}

In this section we give an explicit example of a split group $G$, with unacceptable Langlands dual group, for which the statement of Conjecture \ref{conj main} does not hold. First we introduce the notion of unacceptable homomorphism. 

\begin{df}
Let $\widehat{G}$ be a complex connected reductive algebraic group, and $\varphi:\we_F\to \widehat{G}$ a semi-simple homomorphism. We say that 
$\varphi$ is acceptable if its local conjugacy class is equal to its global conjugacy class. Otherwise we say that $\varphi$ is unacceptable. 
\end{df}

In particular, Theorem \ref{thm acc implies wacc} says that if $\widehat{G}$ is acceptable, then any semi-simple homomorphism $\varphi:\we_F\to \widehat{G}$ is acceptable. Then we have the following proposition.

\begin{prop}\label{prop counter}
Suppose that $G$ is an $F$-split connected reductive group, and that there exists $\varphi_1:\we_F\to \widehat{G}$ which is unacceptable. Let $\varphi_2$ be locally conjugate but not globally conjugate to $\varphi_1$. Now take admissible homomorphisms $\phi_1,\phi_2:\WD_F\to {}^LG$ such that the semi-simple part of the Weil-Deligne parameter attached to $\phi_i$ is $\varphi_i$. Then for any irreducible representation $r$ of ${}^LG$, and any irreducible representation $\phi$ of $\WD_F$, one has \[\gamma(s,r\circ \phi_1  \otimes \phi,\psi)=\gamma(s,r\circ \phi_2 \otimes \phi,\psi).\]
\end{prop}
\begin{proof}
For each $r$ the representations $r\circ \varphi_1$ and $r\circ \varphi_2$ are locally conjugate. Applying \cite[Corollary of Proposition 6, a)]{BbkA8} to the algebra of smooth compactly supported functions on $\we_F$, we deduce that $r\circ \varphi_1$ and $r\circ \varphi_2$ are globally conjugate, as their trace characters are equal. The equality of gamma factors follows. 
\end{proof}

Now we deduce from \cite[Theorem 24]{Weidner}, which we explicate further, the following statement. It is interesting to compare it with 
the remark ending the introduction of \cite{YZ}. 

\begin{prop}
Suppose that $4|q-1$ and that $n\geq 3$. Then there exist non equivalent generic homomorphisms $\phi_1,\phi_2:\WD_F\to {}^L\SO_{2n}(F)$ such that for any irreducible representation $r$ of ${}^LG$, and any irreducible representation $\phi$ of $\WD_F$, one has \[\gamma(s,r\circ \phi_1  \otimes \phi,\psi)=\gamma(s,r\circ \phi_2 \otimes \phi,\psi).\]
\end{prop}
\begin{proof}
Under the condition that $4|q-1$, and by class field theory, the group $(\frac{\BZ}{4\BZ})^2$ is a quotient of $\we_F$ because it is a quotient of $F^\times$. In particular \cite[Theorem 24]{Weidner} provides locally but non globally conjugate parameters $\varphi_i:\we_F\to {}^L\SO_{2n}(F)$, $i=1,2$, by inflation. Let us slightly rephrase and explicate Weidner's example. We denote by $w_n$ the antidiagonal matrix in $\GL_n(\BC)$ with ones on the second diagonal. We set $G=\SO_{2n}(F)$ and realize $\widehat{G}=\SO_{2n}(\BC)$ as the group
\[\SO_{2n}(\BC):=\{M\in \GL_{2n}(\BC),\ w_{2n}M^Tw_{2n}^{-1}=M^{-1}\}.\]
 
We define $\varphi_i: \we_F \to \SO_{2n}(\BC)$ as follows. We write 
\[F^\times=\varpi_F^{\BZ}\times \langle u \rangle \times (1+\varpi_FO_F)\] where $\varpi_F$ is a uniformizer of $F$, $O_F$ is the ring of integers of $F$, and $u$ has order $q-1$. Then 
\[\varphi_i: \we_F \to F^\times \overset{\lambda_i}{\to} \SO_{2n}(\BC),\] where $\lambda_1$ is trivial on $1+\varpi_FO_F$, 
\[\lambda_1(\varpi_F)=\diag(iI_{n-1},I_2,-iI_{n-1})\] and \[\lambda_1(u)=\diag(1,iI_{n-1},-iI_{n-1},1),\] whereas 
\[\lambda_2:=\diag(I_{n-1},w_2,I_{n-1})\lambda_1 \diag(I_{n-1},w_2,I_{n-1})^{-1}.\]

Since \[\lambda_1(\varpi_F^au^b)=\diag(i^a,i^{a+b}I_{n-2},i^b,(-i)^b,(-i)^{a+b}I_{n-2},(-i)^a)\] and \[\lambda_2(\varpi_F^au^b)=\diag(i^a,i^{a+b}I_{n-2},(-i)^b,i^b,(-i)^{a+b}I_{n-2},(-i)^a),\] we see that 
$\varphi_1$ and $\varphi_2$ are locally conjugate in $\SO(2n,\BC)$ by examining the parity of $a$ and $b$:
\begin{itemize}
\item when $b$ is even then $\lambda_1(\varpi_F^au^b)=\lambda_2(\varpi_F^au^b)$,
\item when $b$ is odd and $a$ is even then $\lambda_1(\varpi_F^au^b)$ and $\lambda_2(\varpi_F^au^b)$ are conjugate by the permutation matrix corresponding to the double transposition $(n \ n+1)\circ (1\ 2n)$,
\item when $a$ and $b$ are odd then $\lambda_1(\varpi_F^au^b)$ and $\lambda_2(\varpi_F^au^b)$ are conjugate by the permutation matrix corresponding to the double transposition  $(n \ n+1)\circ (n-1 \ n+2)$.
\end{itemize}
On the other hand, if they were globally conjugate, they would be conjugate by an element in the centralizer of $\lambda_1(\varpi_F)$, which is the set of matrices 
\[\{\diag(A,B,w_{n-1}A^{-T}w_{n-1}),\ A\in \GL_{n-1}(\BC),\   B\in \SO_2(\BC)\}.\] This is impossible since 
$\diag(i,-i)$ is not $\SO_2(\BC)$-conjugate to $\diag(-i,i)$, as $\SO_2(\BC)$ consists of diagonal matrices with determinant one. 
Finally, by \cite[Proposition 6.10]{DHKM}, there exists a monodromy operator $N_i\in \mathcal{N}_{\varphi_i}$ making $(\varphi_i,N_i)$ a generic Weil-Deligne parameter. We then take $\phi_i$ corresponding to $(\varphi_i,N_i)$, and apply Proposition \ref{prop counter}. 
\end{proof}

\begin{rem}
The monodromy operators $N_i$ in the proof above must be equal to zero since the eigenvalues of the adjoint action of 
$(\frac{\BZ}{4\BZ})^2$ on the Lie algebra of $\widehat{G}$ are of absolute value one, whereas the image of the Frobenius element in the group $(\frac{\BZ}{4\BZ})^2$ would act with eigenvalue $q^{-1}$ on $N$ if $N$ was nonzero. 
\end{rem}

\begin{rem}
An interesting question is to know if one can find such examples whenever $\widehat{G}$ is unacceptable, for some non Archimedean local field $F$, or even for all. This has to do with the classification of unacceptable homomorphisms. Namely, the question is: is it possible to find for any unacceptable $\widehat{G}$, an unacceptable parameter $\varphi:\Gamma\to \widehat{G}$ such that $\Gamma$ is a quotient of $\we_F$? 
\end{rem}

\section{A variant}\label{sec var}

Here we consider the following statement. 

\begin{st}\label{st} Let $G=\mathrm{G}(F)$ be a quasi-split group. Then there exists $K/F$ a finite extension and $r:{}^L\mathrm G(K)\to \GL(V)$ a faithful representation of dimension $N$, such that if $\phi_i:\WD_F\to {}^LG$ is a generic homomorphism for $i=1,2$, and \[\gamma(s,r\circ {\phi_1}_{|\WD_K} \otimes \phi,\psi)=\gamma(s,r\circ {\phi_2}_{|\WD_K} \otimes \phi,\psi)\] for all irreducible representation $\phi$ of $\WD_K$ of dimension at most $\lfloor N/2 \rfloor$, then $r\circ \phi_2$ and $r\circ \phi_2$ are conjugate by the normalizer of $r({}^L\mathrm{G}(K))$ inside $\GL(V)$.
\end{st}

The following result is proved for $\mathrm{G}_2(F)$ in \cite{GS} with a different proof. 

\begin{thm}\label{thm variant}
Statement \ref{st} is true for the following list of $G$, $K$ and $r$.
\begin{enumerate} 
\item \label{gr 0} $G=\GL_n(F)$, $K=F$ and $r:\GL_n(\BC)\times \we_F \to \GL_n(\BC)$ is the identity on $\GL_n(\BC)$ extended trivially on $\we_F$, or $G=\U_n(E/F)$, $K=E$ a separable quadratic extension of $F$, and $r:\GL_n(\BC)\times \we_E \to \GL_n(\BC)$ is the identity on $\GL_n(\BC)$ extended trivially on $\we_E$. 
\item \label{gr a} $G=\mathrm{G}_2(F)$, $K=F$, and $r:\mathrm{G}_2(\BC) \times \we_F \to \GL_7(\BC)$ is the standard irreducible representation of $\mathrm{G}_2(\BC)$ extended trivially on $\we_F$. 
\item \label{gr b} $G=\Sp_{2n}(F)$, $K=F$, and $r:\SO_{2n+1}(\BC)\times \we_F\to \GL_{2n+1}(\BC)$ is the natural action of 
$\SO_{2n+1}(\BC)$ on $\BC^{2n+1}$ extended trivially on $\we_F$. 
\item \label{gr c} $G=\SO_{2n+1}(F)$, $K=F$, and $r:\Sp_{2n}(\BC)\times \we_F\to \GL_{2n}(\BC)$ is the natural action of 
$\Sp_{2n}(\BC)$ on $\BC^{2n}$ extended trivially on $\we_F$. 
\item \label{gr d} $G=\SO_{2n+2}(F)$, split form, $K=F$, and $r:\SO_{2n+2}(\BC)\times \we_F\to \GL_{2n+2}(\BC)$ is the natural action of 
$\SO_{2n+2}(\BC)$ on $\BC^{2n+2}$ extended trivially on $\we_F$. 
\item \label{gr e} $G=\SO_{2n+2}(E/F)$, non split quasi-split form, split over the quadratic extension $E/F$, $K=F$, $r:\SO_{2n+2}(\BC)\rtimes \we_F\to \GL_{2n+2}(\BC)$ restricts to 
$\SO_{2n}(\BC)$ as its natural action on $\BC^{2n+2}$, and $r_{|\we_F}$ factors through the Galois group $\Gal_F(E)=\langle \sigma \rangle$ and sends $\sigma$ to an order two element $s\in \O_{2n+2}(\BC)-\SO_{2n+2}(\BC)$. 
\end{enumerate}
\end{thm}
\begin{proof}
The statement is obviously true in the first subcase of Case \eqref{gr 0}, thanks to Theorem \ref{thm loc cv} and Lemma \ref{lm dhkm}. For unitary groups we recall that $\varphi_1\sim \varphi_2$ if and only if ${\varphi_1}_{|\we_E}\sim {\varphi_2}_{|\we_E}$ (\cite[Proposition 7.10]{Prel}). On the other hand, under the assumptions of Statement \ref{st}, we deduce from Theorem \ref{thm loc cv} that ${\varphi_1}_{|\we_E}\sim {\varphi_2}_{|\we_E}$, hence the result now follows from Lemma \ref{lm dhkm}. In Case \eqref{gr a}, the normalizer of $r(\mathrm{G}_2(\BC))$ is equal to $\BC^\times \mathrm{G}_2(\BC)$. This follows for example from applying \cite[Theorem 1.1]{Griess}, and observing that the centralizer of $r(\mathrm{G}_2(\BC))$ is the center of $\GL_7(\BC)$ according to Schur's lemma. Then the theorem in Case \eqref{gr a} follows from Lemma \ref{lm dhkm}, Theorem \ref{thm loc cv} and \cite[Theorem 1.1]{Griess}. In cases \eqref{gr b}, \eqref{gr c}, \eqref{gr d} and 
\eqref{gr e}, the normalizers are respectively $\BC^\times \SO_{2n+1}(\BC)$, $\BC^\times \Sp_{2n}(\BC)$ and $\BC^\times \O_{2n+2}(\BC)$, as follows for example from \cite[Section 2]{Griess}. The result that we claim follows from Lemma \ref{lm dhkm}, Theorem \ref{thm loc cv} and \cite[Theorem 2.3]{Griess}.
\end{proof}

\begin{rem} It is false in general that one can replace the normalizer of $r({}^LG)$ by $r({}^LG)$ itself. For example in Case \eqref{gr e} above, two admissible homomorphisms $\phi_1$ and $\phi_2$ with $\phi_2=s\phi_1 s^{-1}$ for orthogonal $s$ with determinant $-1$ will satisfy the equality of twisted gamma factors, but might be inequivalent. 
\end{rem}

To conclude this paper, we make the following observation for classical groups when $F$ has characteristic zero. Assume the validity of the Gross-Prasad and Rallis conjecture and that the generic local Langlands correspondence $i$ of \cite{JaLiu} coincides with that obtained by endoscopic transfer in \cite{Art} and \cite{Mok}. This latter fact is true at least in the setting of generic supercuspidal representations of symplectic and split orthogonal groups, and we refer to \cite[Theorem B.2 and Remark B.3]{AHKO23}. Then using \cite[Theorem 6.20]{JaLiu} and thanks to the uniqueness of generic representations inside L-packets with respect to a prescribed non degenerate character proved in \cite{Vgen}, Theorem \ref{thm variant} coincides with the local converse theorems mentioned in the introduction, using Langlands-Shahidi gamma factors.\\

	\bibliographystyle{alphanum}
	\bibliography{references}
	
\end{document}